\documentclass{amsart}
\usepackage[T1, T2A]{fontenc}
\usepackage[utf8]{inputenc} 
\usepackage{graphicx}
\usepackage{amssymb,latexsym}
\usepackage{subfigure, url}
\usepackage{tikz}
\usepackage{color}
\def\ov{\overline}
\newcommand{\eps}{\varepsilon}
\newcommand{\Cc}{\mathbb C}
\newcommand{\CP}{\mathbb{C}\mathbb P}
\newcommand{\R}{\mathbb R}
\newcommand{\N}{\mathbb N}

\newcommand{\Hh}{\mathbb H}
\newcommand{\rev}{\mathbb R \rm ev}
\theoremstyle{plain}
\newtheorem{lemma}{Lemma}[section]
\newtheorem{proposition}[lemma]{Proposition}
\newtheorem{theorem}[lemma]{Theorem}

\newtheorem{remark}[lemma]{Remark}
\newtheorem{notation}[lemma]{Notation}
\newtheorem{definition}[lemma]{Definition}
\numberwithin{equation}{section}

\catcode`@=11
\renewcommand\paragraph{\@startsection{paragraph}{4}\z@{5pt}{-\fontdimen 2\font }\bfseries}
\catcode`@=12

\title[Reversible complex polynomial vector fields]{Generic reversible complex polynomial vector fields} \author[C. Rousseau]{Christiane Rousseau}
\address{D\'epartement de math\'ematiques et de statistique, Universit\'e de Montr\'eal, C.P. 6128, Succursale Centre-ville, Montr\'eal (QC), H3C 3J7, Canada.}
\email{christiane.rousseau@umontreal.ca}

\usepackage[hidelinks]{hyperref}

\begin{document} 
\date{November 13 2024}
\maketitle

\begin{abstract} The paper studies the generic complex 1-dimensional polynomial vector fields of the form $iP(z)\frac{\partial}{\partial z}$, where $P$ is a polynomial with real coefficients, under topological orbital equivalence preserving the separatrices of the pole at infinity. The number of  generic strata is determined, and a complete parametrization of the strata is given in terms of a modulus formed by a combinatorial part and an analytic part. The bifurcation diagram is described for degrees 3 and  4. A realization theorem is proved for any generic modulus. \end{abstract}

\section{Introduction} The study of the family $\mathcal{P}_{\Cc, k+1}$ of complex 1-dimensional polynomial vector fields of the form $P_\eps(z) \frac{\partial}{\partial z}$, where \begin{equation}P_\eps(z) = z^{k+1} + \eps_{k-1} z^{k-1} + \dots + \eps_1z+\eps_0\label{eq:P}\end{equation} and $\eps\in \Cc^k$  has been initiated by Douady, Estrada and Sentenac in the groundbreaking paper \cite{DES05}. The emphasis is put there on the generic (i.e. structurally stable) vector fields inside the family $\mathcal{P}_{\Cc, k+1}$. The generic vector fields have simple singular points and no homoclinic loop through the pole at $\infty$. There are $C(k) = \frac{\binom{2k}{k}}{k+1}$ strata of  generic vector fields.  The description of the generic strata  proved to be  essential in the  construction of  a modulus of analytic classification for generic $k$-parameter unfoldings $g_\eps$ of parabolic germs $g_0$ of codimension $k$, i.e. germs of 1-dimensional diffeomorphisms $g_0: (\Cc,0)\rightarrow (\Cc,0)$ with a multiple fixed point of multiplicity  $k+1$ (also called a \emph{parabolic point}) (see for instance \cite{MRR04}, \cite{Ri08} et \cite{Ro15}). Indeed a formal normal form for a generic unfolding of a parabolic germ $g_\eps$ is given by the time-one map $v_\eps^1$ of a vector field $v_\eps=\frac{P_\eps(z) }{1+a(\eps)z}\;\frac{\partial}{\partial z}$, whose phase portrait is very close to that of the polynomial vector field $P_\eps(z) \frac{\partial}{\partial z}$ for small $z,\eps$. 

There are particular families $g_\eps$ unfolding a parabolic germ which were studied in greater detail. One such family is the square $g_\eps=f_\eps\circ f_\eps$ of a generic unfolding $f_\eps$ of an antiholomorphic parabolic germ of codimension $k$ (see for instance \cite{GR23} and \cite{Ro23b}). In that particular case, the parameter $\eps$ naturally belongs to 
$\R^k$, and a formal normal form for a generic unfolding $f_\eps$ is given by $\sigma\circ v_\eps^{\frac12}$, where $\sigma(z) =\ov{z}$, $v_\eps=\frac{P_\eps(z) }{1+a(\eps)z}\;\frac{\partial}{\partial z}$, and $$P_\eps(z) = \pm z^{k+1} + \eps_{k-1} z^{k-1} + \dots + \eps_1z+\eps_0,$$ 
has real coefficients.
Note that the two cases $\pm$ are equivalent under $z\mapsto -z$ when $z$ is odd. When $k$ is even, the minus case can be obtained from the plus case by changing the time, i.e. through the change $(z,t,\eps)\mapsto (z,-t,-\eps)$.  The study of generic vector fields with real coefficients appears in \cite{GR24}: there are two types of strata of generic vector fields: strata of vector fields which are generic inside the family $\mathcal{P}_{\Cc, k+1}$, and strata for which the real axis (which is invariant) is a homoclinic loop through $\infty$. 

Other natural particular families $g_\eps$ unfolding a parabolic germ are the \emph{reversible} ones: they satisfy 
\begin{equation} \sigma\circ g_\eps= g_\eps^{-1} \circ \sigma,\label{diffeo_reversible}\end{equation}
where $\sigma(z) =\ov{z}$. Such parabolic germs occur naturally in two contexts. 

The first context is the  classification of  generic unfoldings of 2-dimensional real resonant saddles of codimension $k$, up to orbital equivalence. It has been shown by Martinet and Ramis \cite{MR83} that two germs of vector fields with resonant saddles having the same ratio of eigenvalues are orbitally equivalent if and only if the holonomies of the corresponding separatrices are conjugate, and it is straightforward to generalize the theorem to unfoldings. The holonomies of the separatrices of an unfolding of a real saddle are reversible. A formal orbital normal form for a resonant saddle is given by 
$$px\frac{\partial}{\partial x} +y \left(-q\pm\frac{u^k}{1+Au^k}\right),$$
for $u=x^qy^p$ and $A\in \R$, and a formal normal form for a generic unfolding is given by
$$px\frac{\partial}{\partial x} +y \left(-q+\frac{\sum_{j=0}^{k-1}\eps_ju^j\pm u^k}{1+A(\eps)u^k}\right),$$
with $\eps\in \R^k$ and $A(\eps)\in \R$. For $\eps=0$, the holonomy $h_{x,\eps}$ of the $x$-axis (resp. $h_{y,\eps}$ of the $y$-axis) has a parabolic point if $p=1$ (resp. $q=1$). Let $w_\eps$ be the vector field $w_\eps=\frac{2\pi i}{q}\; \frac{Q_\eps(y)}{1+A(\eps)y^k}
\;\frac{\partial}{\partial y}$, where
$$Q_\eps(y) =  y( \pm y^k+ \eps_{k-1} y^{k-1} + \dots + \eps_1y+\eps_0).$$
The normal form of  $h_{x,\eps}$ is then  $w_\eps^1$.
Using real scalings  in $y$ and in the parameters, we can rather consider vector fields  $i\frac{Q_\eps(y)}{1+A(\eps)y^k}
\;\frac{\partial}{\partial y}$, which are reversible and whose phase portraits in a fixed disk and small $\eps$ resemble a lot to those of $iQ_\eps(y)\frac{\partial}{\partial y}$. 

The second context is the study of the conformal classification of generic unfoldings of curvilinear angles defined by two germs of real analytic curves $\gamma_1$ and $\gamma_2$ intersecting at the origin with an angle $\theta$ (see for instance \cite{AG05} and \cite{Ro07}). If $\Sigma_j$ is the Schwarz reflection with respect to $\gamma_j$, then $g=\Sigma_2\circ \Sigma_1$ is a germ of holomorphic diffeomorphism having a fixed point at the origin with a multiplier $\exp(2 i \theta)$, and satisfying $\Sigma_j\circ g= g^{-1}\circ\Sigma_j$ for $j=1,2$. A conformal change of coordinates can bring $\Sigma_2$ to $\sigma$, and then $g$ and its generic unfoldings $g_\eps$ satisfy \eqref{diffeo_reversible}. In the particular case $\theta=0$, the map $g$ has a parabolic point at the origin.

These examples of reversible unfoldings of parabolic germs are the main motivation for the  study of generic reversible polynomial vector fields of the form $v_\eps(z)=iP_\eps(z)\frac{\partial}{\partial z}$, where \begin{equation}P_\eps(z) = \pm z^{k+1} + \eps_{k-1} z^{k-1} + \dots + \eps_1z+\eps_0\label{eq:P2}\end{equation}
and $\eps\in \R^k$, which is the purpose of this paper. We limit ourselves to the plus sign in \eqref{eq:P2} and we call $\mathcal{P}_{\rev, k+1}$ the corresponding class of vector fields. (The  minus sign case  can be obtained from the plus sign by reversing the time.) The generic vector fields are defined as the structurally stable ones. The main differences with the generic vector fields in $\mathcal{P}_{\Cc,k+1}$ studied in \cite{DES05} is that all simple real singular points are centers and that homoclinic loops through infinity which are symmetric with respect to the real axis are structurally stable. In the paper we determine the exact number of generic strata. We describe each generic stratum through 
\begin{itemize}
\item a combinatorial invariant given by a non-crossing involution on $\{0, \dots, k\}$, which preserves intervals between fixed points; 
\item and an analytic invariant, which is an element of $\eta\in (\R^+)^a\times (i\R^+)^{b} \times \Hh^{c}$, for some appropriate $a,b,c\in \N$, such that $a+b+2c=k$. 
\end{itemize} We then show that each combinatorial type and element $\eta\in (\R^+)^a\times (i\R^+)^{b} \times \Hh^{c}$, for appropriate $a,b,c\in \N$  determined by the combinatorial type (and satisfying $a+b+2c=k$), can be realized by a vector field in $\mathcal{P}_{\rev,k+1}$.

\section{Preliminaries} 
\begin{notation} \begin{enumerate}
\item We call $\mathcal{P}_{\rev,k+1}$ the family of polynomial vector fields of the form $v_\eps(z)=iP_\eps(z)\frac{\partial}{\partial z}$ where $P_\eps(z)$ is given in \eqref{eq:P}
and $\eps\in \R^k$.
\item The corresponding family with $\eps\in\Cc$ 
is called $i\mathcal{P}_{\Cc,k+1}$.
\item The family of vector fields $v_\eps(z)=P_\eps(z)\frac{\partial}{\partial z}$ where $\eps\in\Cc$ is called $\mathcal{P}_{\Cc,k+1}$.
\end{enumerate}
 \end{notation}

\begin{proposition} \label{rem:scaling} \begin{enumerate} 
\item Vector fields $v_\eps(z)=iP_\eps(z)\frac{\partial}{\partial z}$, where $P_\eps$ has real coefficients, have the property that $iP_\eps(\ov{z})= -\ov{i P_\eps(z)}$, i.e. the phase portrait is reversible, and all trajectories through real nonsingular points are perpendicular to the real axis. 
\item The singular point at infinity is a pole of order $k-1$. It has $k$ attracting separatrices in the directions $\exp\left(-i\frac{\pi}{2k}+j\frac{2\pi}{k}\right)$, $j=0, \dots, k-1$, and $k$ repelling separatrices in the directions $\exp\left(i\frac{\pi}{2k}+j\frac{2\pi}{k}\right)$, $j=0, \dots, k-1$ (see Figure~\ref{infinity}).
\item The study of $iP_\eps(z) \frac{\partial}{\partial z}$ for 
\begin{equation}P_\eps(z) = -z^{k+1} + \eps_{k-1} z^{k-1} + \dots + \eps_1z+\eps_0\label{Pneg}\end{equation} 
with $\eps\in \R^k$ can be obtained from that of $v_{-\eps}(z)$ through $(z, \eps, t)\mapsto (z, -\eps, -t)$.
\item For $k$ odd, the study of $iP_\eps(z) \frac{\partial}{\partial z}$, with $P_\eps$ given in \eqref{Pneg} and $\eps\in \R^k$ can be obtained from that of $v_{\eps^*}(z)$ through $(z, \eps, t)\mapsto (-z, \eps^*, t)$ with  $\eps_j^*=(-1)^{j+1} \eps_j$.
\item For $r\in\R^+$, the change \begin{equation}(z, \eps_{k-1}, \dots, \eps_1,\eps_0,t)\mapsto (zr, \eps_{k-1} r^{-(k-2)}, \dots, \eps_1,\eps_0r,tr^k)\label{scaling}\end{equation} sends a vector field to one with the same phase portrait modulo a zoom.  This allows using scalings when discussing particular situations. 
\item The change $z\mapsto \exp\left(\frac{\pi i}{2k}\right)z$ maps $v_\eps(z)=iP_\eps(z)\frac{\partial}{\partial z}$ to $P_{\eps'}(z)\frac{\partial}{\partial z}$ for $\eps_j'= \exp\left(\frac{\pi i(1-j)}{2k}\right)\eps_j$.\end{enumerate} 
\end{proposition}
\begin{figure}\begin{center}\includegraphics[width=5cm]{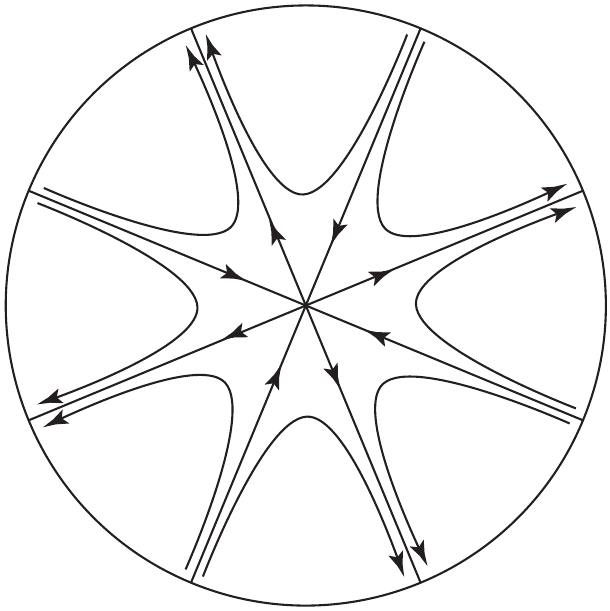}\caption{The separatrices of infinity when $k=4$.}\label{infinity}\end{center}\end{figure}

\begin{definition} The union of the separatrices of infinity, called the \emph{separatrix graph}, divides the plane into simply connected regions called \emph{zones}. When a vector field of $\mathcal{P}_{\rev,k+1}$ has only simple singular points, the zones can be of two types:
\begin{itemize} 
\item  \emph{rotation zones} around a center limited by separatrices of infinity;
\item \emph{ $\alpha\omega$-zones}, which are formed by the  union of all trajectories from one repelling node or focus to one attracting node or focus.
An $\alpha\omega$-zone may contain homoclinic loops in its boundary. \end{itemize}
\end{definition}

\subsection{Generic vector fields}

\begin{definition} A vector field $v_\eps(z)\in \mathcal{P}_{\rev,k+1}$ (resp. $v_\eps(z)\in \mathcal{P}_{\Cc,k+1}$) is called \emph{generic} if it is structurally stable inside the family $\mathcal{P}_{\rev,k+1}$ (resp. $\mathcal{P}_{\Cc,k+1}$),  i.e. all vector fields $v_{\eps'}(z) \in \mathcal{P}_{\rev,k+1}$ (resp. $v_{\eps'}(z) \in \mathcal{P}_{\Cc,k+1}$) with $\eps'\in \R^k$ (resp. $\eps'\in \Cc^k$) and $|\eps-\eps'|<\delta$ for some $\delta>0$ are topologically orbitally equivalent to 
$v_\eps(z)$, where the orbital equivalence fixes the separatrices at infinity. \end{definition} 

One purpose of the paper is to describe the equivalence classes of generic vector fields in $\mathcal{P}_{\rev,k+1}$, called \emph{generic strata} of $\mathcal{P}_{\rev,k+1}$. This can be done through partitioning the parameter space into open sets of structurally stable vector fields, and non generic vector fields occurring on the bifurcation set. In practice, when describing the equivalence classes of vector fields, we identity them with the set of parameter values $\eps$ parametrizing the vector fields. 

Generic vector fields in $\mathcal{P}_{\Cc,k+1}$ have been characterized and described geometrically by Douady, Estrada and Sentenac in the paper \cite{DES05}. They will be called \emph{DES-generic}. More precisely,

\begin{theorem}\label{thm:DES1} \cite{DES05} \begin{enumerate} \item A vector field $v_\eps(z)\in \mathcal{P}_{\Cc,k+1}$ is DES-generic if and only if all its singular points are simple and there is no homoclinic loop through $\infty$.
\item Let $v_\eps(z)\in \mathcal{P}_{\Cc,k+1}$ be a DES-generic vector field. The union of the separatrices of $\infty$ (called the \emph{separatrix graph}) divides  $\CP^1$ into $\alpha\omega$-zones. Each zone is adherent to two singular points, one attracting, one repelling. 
\end{enumerate} \end{theorem}

\begin{theorem}\label{thm:DES2} \cite{DES05} \begin{enumerate}\item There are $C(k)=\frac{\binom{2k}{k}}{k+1}$ strata of generic vector fields inside $\mathcal{P}_{\Cc,k+1}$, each characterized by a combinatorial type, which can be described in any of the following ways:
\begin{enumerate}
\item  There are $2k$ ends at infinity labelled $e_{\pm1}, \dots, e_{\pm k}$ as in Figure~\ref{zones_DES}. Each $\alpha\omega$-zone contains two \emph{ends} at infinity. This induces a non-crossing involution $\tau$ with no fixed points on $\{\pm1, \dots, \pm k\}$, sending one end of a zone to the other end of the same zone. 
\item Two  trajectories linking two singular points are \emph{equivalent} if they go  from the same repelling point to the same attracting point. An equivalence class of trajectories is an edge of a graph whose vertices are the singular points. The graph is a tree. The combinatorial type is the union of the tree and its attachment to the separatrices. 
\end{enumerate} 
\item Each $\alpha\omega$-zone contains an oriented trajectory of $e^{i\beta}v_\eps(z)$ for some $\beta\in (0,\pi)$, which goes from one end of the zone to the other end. Since $e^{i\beta}\in \Hh$, the complex travel time of $v_\eps(z)$ along the trajectory is an element of $\Hh$. This complex travel time is called the \emph{transversal time} of the zone.
\item  A generic vector field has $k$ $\alpha\omega$-zones. It is possible to parametrize the vector fields of a generic stratum through the vectors  $\eta=(\eta_1, \dots, \eta_k)\in \Hh^k$, whose coordinates are the $k$ transversal times.  
\item Any combinatorial type and vector $\eta\in\Hh^k$ can be realized by a unique generic vector field $v_\eps\in \mathcal{P}_{\Cc,k+1}$. \end{enumerate}
\end{theorem}
\begin{figure} \begin{center} \includegraphics[width=6cm]{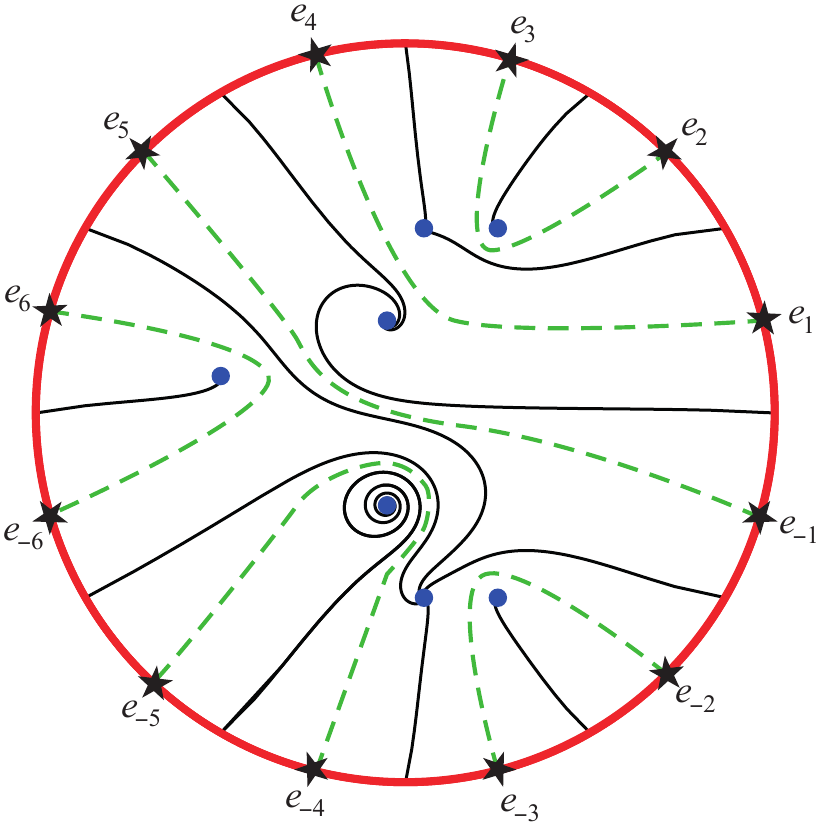}\caption{The $\alpha\omega$-zones are bounded by the black separatrices. The ends $e_j$ correspond to the directions at $\infty$ given by $\exp\left(\frac{\pi(2j-{\rm sgn}(j)}{2k}\right)$, $j=\pm1, \dots, \pm k$. Each transversal time is the time from $\infty$ to $\infty$  along a dotted green line which is a trajectory of $e^{i\alpha}v_\eps(z)$ for some $\alpha\in \R$. The involution (defined on the set of indices) sends one end $e_j$ of a zone to the other end of the zone. For this particular example: $\tau(1)=4$, $\tau(2)=3$, $\tau(5)=-1$, $\tau(6)=-6$, $\tau(-2)=-3$, $\tau(-4)=-5$.}\label{zones_DES}\end{center}\end{figure}

\section{Generic vector fields of $\mathcal{P}_{\rev, k+1}$}

\begin{proposition} Let $v_\eps(z)= iP_\eps(z) \frac{\partial}{\partial z}\in \mathcal{P}_{\rev,k+1}$ be a generic vector field. Then (see Figure~\ref{phaseportrait})
\begin{enumerate} 
\item All singular points are simple.
\item The real singular points are centers and all centers are real. The basin of a center is called a \emph{rotation zone}. The boundary of a rotation zone is 
either 
\begin{itemize}
\item one homoclinic loop through infinity containing one end of the real axis;
\item or  two homoclinic loops through infinity (see Figure~\ref{phaseportrait}(a)).\end{itemize}
\item The only homoclinic loops through $\infty$ are symmetric with respect to the real axis: they join a pair of separatrices $s_{\pm j}$. 
\item If there exist $h$ homoclinic loops, then these divide the phase plane into $h+1$ simply connected domains that we call \emph{regions}. Each region either contains a unique singular point which is a center located on the real axis, or it contains an even number of complex singular points whose sum of the periods is real.  
\item A region can be either the whole plane, or its boundary can have one or two homoclinic loops. 
\item Consider a region containing an even number $2\mu$ of complex singular points. We define the following graph attached to this region: 
\begin{itemize} \item The vertices are the singular points inside  the region. They are attracting or repelling nodes or foci.
\item An edge between two vertices, one repelling, one attracting,  is an equivalence classes of trajectories with $\alpha$-limit set at the repelling vertex and   $\omega$-limit set at the attracting vertex. \end{itemize} 
Then this graph is a tree, which is symmetric with respect to the real axis. It has exactly one edge which crosses the real axis. 
The region contains $4\mu-2$ separatrices of $\infty$ attached to the singular points (see Figure~\ref{phaseportrait}(b)).
\end{enumerate} 
The conditions \emph{(1)-(3)} are sufficient for a vector field $v_\eps(z)\in \mathcal{P}_{\rev,k+1}$ to be generic.
\end{proposition}
\begin{proof} The proof is straighforward considering the fact that simple real singular points are structurally stable, as well as homoclinic loops symmetric with respect to the real axis. For the latter case, this comes from the fact that the number and type (real or complex) of singular points on each side of the loop is structurally stable and that the sums of their periods on each side of the loop remain pure imaginary.\end{proof}
\begin{figure}\begin{center} 
\subfigure[]{\includegraphics[width=5cm]{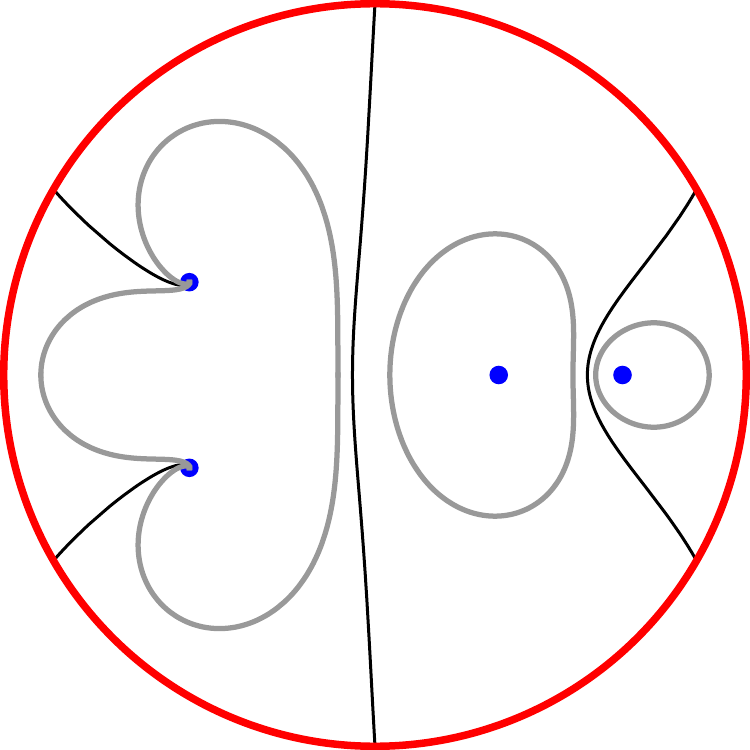}}\qquad \subfigure[]{\includegraphics[width=5cm]{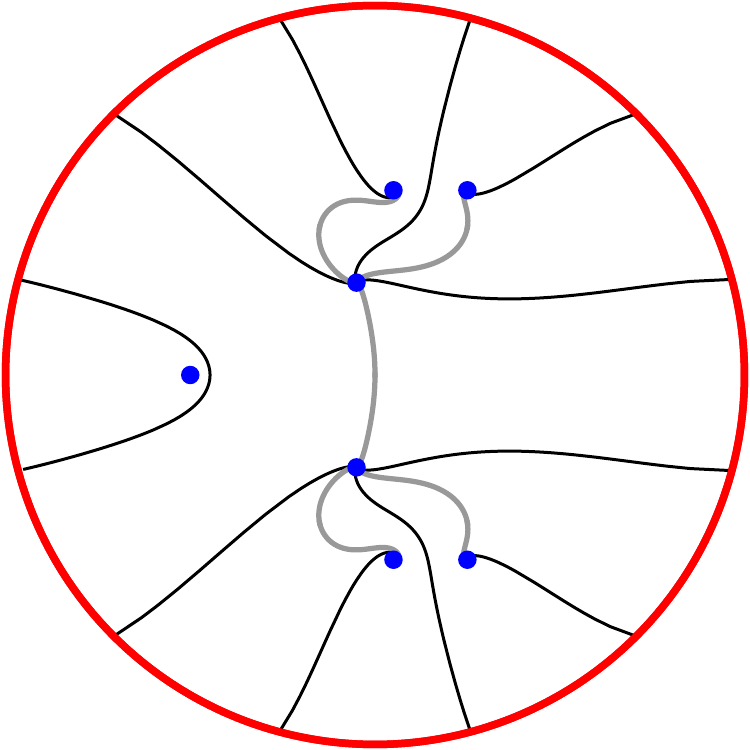}}\caption{(a) Rotation zones bounded by one or two homoclinic loops. (b) A rotation zone and a region with six singular points and ten separatrices landing at the singular points. }\label{phaseportrait}\end{center}\end{figure}

\begin{definition}\label{def:extremal} Consider a region containing an even number $2\mu$ of complex singular points. 
\begin{itemize} 
\item The region is called \emph{central} if it is bounded by two homoclinic loops (see Figure~\ref{extremal_central}(a)). The $\alpha\omega$-zone containing a segment of the real axis is called a \emph{central zone}.
\item The region is called \emph{left-extremal} (resp. \emph{right-extremal}) if it is bounded by one homoclinic loop and contains the negative (resp. positive) end of the real axis (see Figure~\ref{extremal_central}(b) and \ref{extremal_central}(c)). The $\alpha\omega$-zone containing a right end (resp. left end) of the real axis is called a \emph{right-extremal zone} (resp. \emph{left-extremal zone}).
\item The region is called \emph{bi-extremal} if it contains the whole real axis (see Figure~\ref{extremal_central}(d)). The zone containing the real axis is called a \emph{bi-extremal zone}.
\end{itemize}\end{definition}
\begin{figure}\begin{center} 
\subfigure[A central region]{\includegraphics[width=2.6cm]{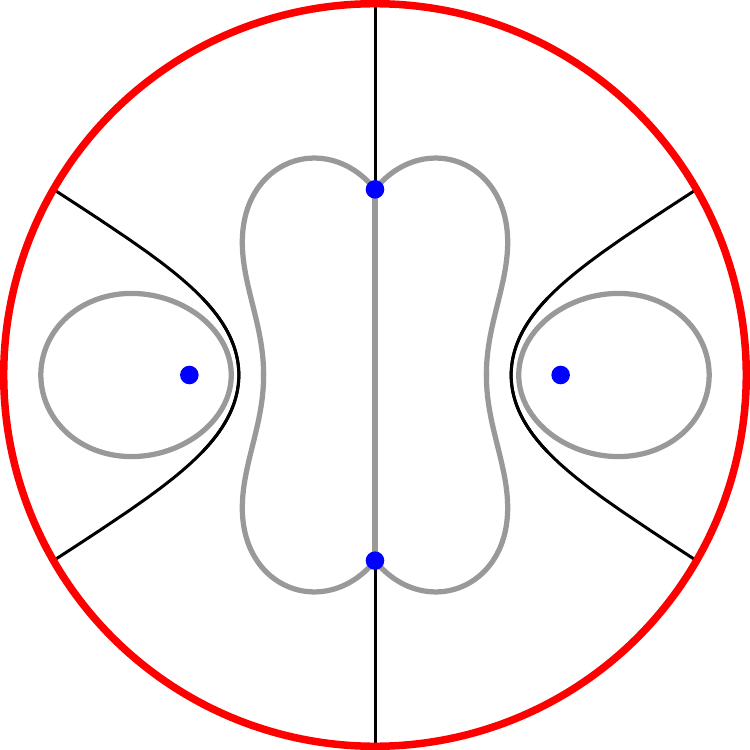}}\quad \subfigure[A right-extremal region]{\includegraphics[width=2.6cm]{exemple1}}\quad \subfigure[A left-extremal region]{\includegraphics[width=2.6cm]{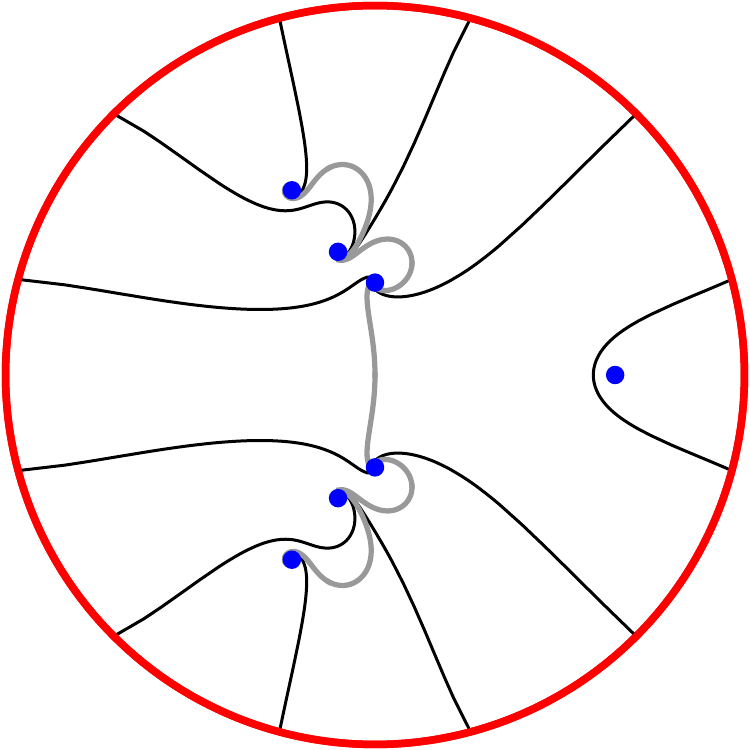}}\quad \subfigure[A bi-extremal region]{\includegraphics[width=2.6cm]{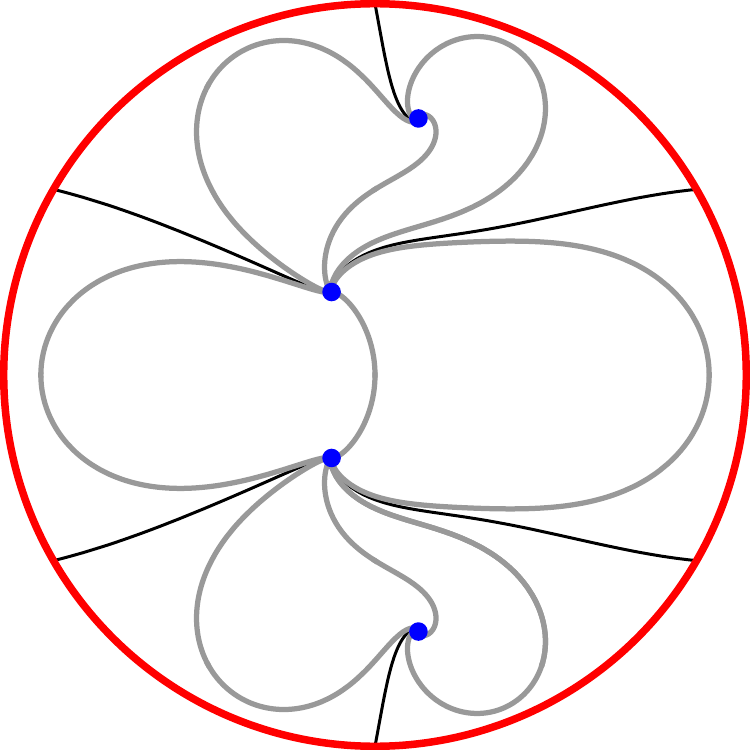}}\caption{Extremal and central regions. }\label{extremal_central}\end{center}\end{figure}

\begin{theorem}\label{thm:top_equiv}  The topological orbital equivalence classes (preserving the separatrices at infinity) of generic vector fields of $\mathcal{P}_{\rev,k+1}$  are in bijection with the non-crossing involutions 
$\tau$ on $\{0, 1, \dots,k\}$, which preserve intervals between fixed points. \end{theorem}
\begin{proof} Let $s_j$, $j=\pm1, \dots, \pm k$, be the separatrices in the directions $\exp\left(\frac{\pi i (2j-{\rm sgn}(j)}{2k}\right)$, and let $e_j$, $j=0, \pm 1, \pm (k-1),k,$ be the ends at $\infty$ in the directions $\exp\left(\frac{\pi i j}{k}\right)$. The involution is defined on the ends at infinity and uses the symmetry with respect to the real axis. 
The basin of a real singular point (of center type) is either 
\begin{itemize}
\item surrounded by one homoclinic connection between the separatrices $s_{\pm1}$ or $s_{\pm k}$: in this case $\tau(0)=0$ or $\tau(k)=k$;
\item or by a pair of homoclinic connections, one  between $s_j$ and $s_{-j}$ and the other one between one  between $s_{j+1}$ and $s_{-(j+1)}$, for $1\leq j\leq k-1$: then the ends $e_{\pm j}$ belong to the basin of the center and we let $\tau(j)=j$. \end{itemize} 

Outside the rotation zones, there are the regions containing $2\ell$ complex singular points. The $\alpha\omega$-zones not touching the real axis occur in pairs, one in the upper half-plane, and one in the lower half-plane. The involution is defined for the upper $\alpha\omega$-zone: it sends one end of the zone to the other end of the zone. 

Let us now onsider the zones containing a portion of the real axis. 

A bi-extremal region only occurs in a DES-generic vector field for $k$ odd and it contains all $k+1$ singular points.
Then the involution is defined by $\tau(0)=k$ and $\tau$ restricted to ${\{1, \dots, k-1\}}$ is a non-crossing involution without fixed points. 

A right-extremal or left-extremal region with $2\ell$ singular points contains $2\ell-1$ separatrices in each of the lower and upper half-planes and $4\ell-1$ ends at infinity, including one end on the real axis.
It is bounded by a separatrix connection between $s_j$ and $s_{-j}$ where $j=2\ell$ (resp. $j=k-2\ell+1$) for a right-extremal (resp. left-extremal) region. Inside the region, the number of $\alpha\omega$-zones is odd, and one $\alpha\omega$-zone is right-extremal (resp. left-extremal). If the zone is right-extremal, then the involution sends $\{0, \dots, 2\ell-1\}$ to itself and is defined as 
$$\begin{cases} 
\tau(0)=2\ell-1,\\
\tau(j)\neq j,&j\in\{1, \dots, 2\ell-2\}.\end{cases} $$

The upper and lower parts of a central region with $2\ell$ singular points each contain $2\ell-1$ separatrices and $2\ell$ ends. For the upper part, these are the ends $e_n, e_{n+1}, \dots, e_{n+2\ell-1}$, with $1\leq n\leq k-2\ell$ and we change the sign of indices for the lower part. The number of $\alpha\omega$-zones delimited by the separatrix graph is odd: 
\begin{itemize} 
\item the central $\alpha\omega$-zone containing a segment of the real axis: for this zone we have $\tau(n)= n+2\ell-1$;
\item symmetric pairs of $\alpha\omega$-zones with respect to the real axis. The upper one is such that $\tau(j)= j' $ with $j\neq j'$, $n+1\leq j,j'\leq n+2\ell-2$.
\end{itemize}
The map sending equivalence classes to non-crossing involutions on $\{0, 1, \dots,k\}$ which preserve intervals between fixed points is injective. To show that it is surjective we need to show that any involution can be realized. We postpone this to Theorem~\ref{thm_realization} in Section~\ref{sec:realization}. 
\end{proof} 

\begin{definition}\label{def:comb_inv} The non-crossing involution $\tau$ on $\{0, 1, \dots,k\}$ defined in Theorem~\ref{thm:top_equiv} and which  preserves intervals between fixed points is called the \emph{combinatorial invariant} of a generic stratum of $\mathcal{P}_{\rev,k+1}$. It is also called the combinatorial invariant of any generic vector field in the stratum.\end{definition}

It is possible to recover the attachment of the separatrices to the different singular points from the combinatorial invariant. From the reversibility, it suffices to describe what happens in the upper half-plane. The non real singular points will be recovered as an equivalence class of separatrices landing at the same point.  The details are as follows. 

\begin{proposition}\label{prop:attachment} Let be given a generic vector field in $\mathcal{P}_{\rev,k+1}$ and its associated non-crossing involution $\tau$ on $\{0, 1, \dots,k\}$, which  preserves intervals between fixed points. We decompose $\{0, \dots, k\}$ as a disjoint union $A_1\cup \dots \cup A_q$ of nonvoid subsets such that 
\begin{itemize} 
\item the $A_j$ are \lq\lq ordered\rq\rq, namely if $j<j'$, then all elements of $A_j$ are smaller than any element of $A_{j'}$;
\item the $A_j$ are invariant under $\tau$; 
\item the $A_j$ are minimal, i.e. contain no strict subsets also invariant under $\tau$. 
\end{itemize} 
 Let $a_j$ be the smallest element of $A_j$. Let $s_\ell$ be the separatrix in the direction $\exp\left(\pi i \frac{2\ell-{\rm sgn}(\ell)}{2k}\right)$, $\ell=\pm 1, \dots, \pm \ell$. Then 
\begin{enumerate}
\item Each $A_j$ either contains one element, which is a fixed point of $\tau$,  or contains an even number of elements, none of them fixed by $\tau$. 
\item The number $m$ of real singular points of the vector field is equal to the number of fixed points of $\tau$, which  is equal to the number of $A_j$ containing exactly one element. 
\item The following homoclinic loops symmetric with respect to the $\R$-axis separate the different regions or rotation zones:
$$s_{a_j} = s_{-a_j}, \qquad \text{for \ } j=2, \dots, q.$$ Hence, $q=h+1$, where $h$ is the number of homoclinic loops of the vector field.

\item If $A_j= \{a_j, a_{j}+1, \dots, a_{j}+2n-1= a_{j+1}-1\}$ contains an even number of elements, then the separatrices $s_{a_j+1}, \dots, s_{a_{j+1}-1}$ lie in the interior of the region (the region has the ends $e_{a_j}, e_{a_j+1}\dots, e_{a_{j+1} -1}$). We define the following permutation on $\{a_j+1, \dots, \allowbreak a_{j+1}-1\}$:
$$\sigma_j(h)=\tau (h-1), \qquad h\in\{a_j+1, \dots, a_{j+1}-1\}.$$
For each $h$, let $p$ be the smallest integer such that $\sigma_j^{\circ p}(h)=h$. Then $[h]=\{h, \sigma_j(h), \dots \sigma_j^{\circ p-1} (h)\}$ is the equivalence class of $h$ and all separatrices with index in $[h]$ land at the same singular point called $z_{[h]}$. 
\end{enumerate} 
\end{proposition}
\begin{proof} The proof is straightforward and the details similar to \cite{DES05}. The only thing to prove is that the image of $\sigma_j$ is in $\{a_j+1, \dots, \allowbreak a_{j+1}-1\}$, i.e. does not take the value $a_j$. 
This comes from the fact that $a_j=\tau(a_{j+1}-1)$ and $a_{j+1}$ is not in the domain of $\sigma_j$.
\end{proof}

\subsection{Number of generic strata}

\begin{theorem}\label{thm:number-strata} The number $G(k)$ of generic strata of vector fields in $\mathcal{P}_{\rev,k+1}$  is
given by $G(k)=\binom{k+1}{\lfloor\frac{k+1}2\rfloor}$. The numbers $G(k)$ (starting at $k=-1$) are those of the integer sequence A001405. The first numbers appear in Table~\ref{table1}. 
\begin{table}[h] \begin{center} \begin{tabular} {|| c | c|c|c|c |c|c|c|c|c |c|c|c|c||}
\hline
$k$ & -1&0&1&2&3&4&5&6&7&8&9&10&11\\
\hline
$G(k)$&1&1&2&3&6&10&20&35&70&126&252&462&924\\
\hline \end{tabular}\end{center}\caption{The first values of $G(k)$.}\label{table1} \end{table}
\end{theorem}
\begin{proof} The proof follows from Lemma~\ref{lemma:bijection} and Proposition~\ref{numberDDP} below.\end{proof}

\begin{definition} \begin{enumerate} 
\item A \emph{Dyck path} is a lattice path on $\N^2$ starting at $(0,0)$, ending at $(n,0)$, with steps of the form $(1,1)$, $(1,-1)$.

\item  A \emph{ dispersed Dyck path of length $n$}  is a path in $\N^2$ from $(0,0)$ to $(n,0)$ with steps $(1,1)$, $(1,-1)$ and $(1,0)$, and with no steps $(1,0)$ at positive height. 
 \end{enumerate}\end{definition}

\begin{lemma}\label{lemma:bijection} There is a bijection with  non-crossing involutions 
$\tau$ on $\{0, 1, \dots,k\}$, which preserve intervals between fixed points and dispersed Dyck paths of length $k+1$.\end{lemma}
\begin{proof} To each non-crossing involution $\tau$ on $\{0, 1, \dots,k\}$, which preserves intervals between fixed points, 
 we associate a dispersed Dyck path with $k+1$ steps, where the $j$-th step, for $j=0, \dots, k$, is given by: 
$$\begin{cases} (1,1), &\tau(j)>j,\\
(1,-1), &\tau(j)<j,\\
(1,0), &\tau(j)=j.\end{cases}$$
This lattice path starts at $(0,0)$ and ends at $(k+1,0)$ (see Figure~\ref{Dyck}). Because $\tau$ is non-crossing then it remains in $\N^2$. Moreover, because $\tau$ preserves intervals between its fixed points, then all steps $(1,0)$ occur at level $0$. 

Conversely consider a dispersed Dyck path of length $k+1$, whose steps are numbered $0, \dots, k$, the numbering of a step being called its \emph{order}.  We associate to this path an involution $\tau$ on $\{0, 1, \dots, k\}$. 
Then $\tau(j):= j$ if the $j$-th step is $(1,0)$. Let $j<j'$ be two consecutive fixed points. Then $\tau(j'')$ for $j <j''<j'$ is defined by increasing order of $j''$: if $j''$ is the order of a step of type $(1,1)$ (resp. $(1,-1)$), then $\tau(j'')$ is the order of the first step of the form $(1,-1)$ to the right of $j''$ (resp. of the form $(1, 1)$ to the left of $j''$) and occuring at the same height (i.e. within the same horizontal strip of vertical width equal to 1). The involution $\tau$ is then non-crossing and preserves intervals between fixed points. 
 \end{proof} 
\begin{figure} \begin{center} \subfigure[]{\includegraphics[width=2.4cm]{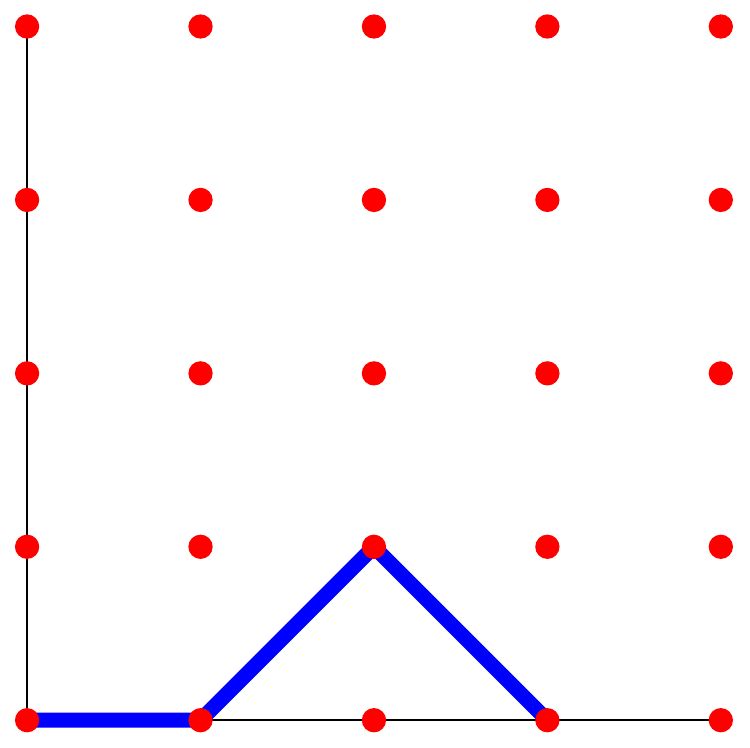}}\quad\subfigure[]{\includegraphics[width=4.6cm]{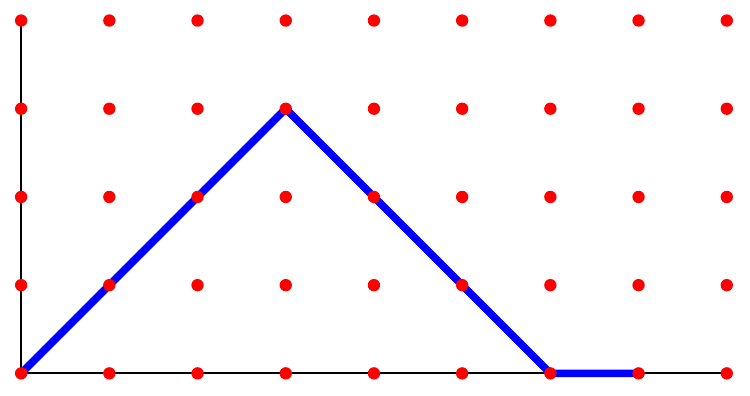}}\quad \subfigure[]{\includegraphics[width=4.6cm]{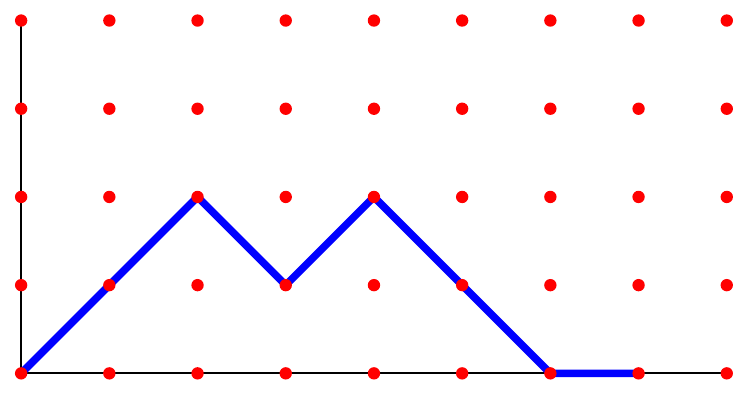}}\caption{The dispersed Dyck paths associated to the vector fields of Figure~\ref{extremal_central}.}\label{Dyck}\end{center}\end{figure}

The number of dispersed Dyck paths of length $n$ is well known in the literature. We provide a proof for the sake of completeness. 

\begin{proposition}\label{numberDDP} The number of dispersed Dyck paths of length $n$ is given by $\binom{n}{\lfloor\frac{n}2\rfloor}$.
\end{proposition}
\begin{proof} Let $N=\begin{cases} 0,& n\: \text{even},\\
-1,&n\: \text{odd}.\end{cases}$. Let us show that the number of dispersed Dyck paths of length $n$ is equal to the number of paths of length $n$ from $(0,0)$ to $(n,N)$ with steps $(1,1)$ and $(1,-1)$. This latter number is obviously equal to the number of choices for the steps $(1,1)$, namely $\binom{n}{\lfloor\frac{n}2\rfloor}$. The bijection is constructed as follows. 

Consider a dispersed Dyck path. Let the steps $(1,0)$ occur between $(n_i,0)$ and $(n_i +1,0)$, $i= 1, \dots, \ell$. A step $(1,0)$ occuring between $(n_i,0)$ and $(n_i +1,0)$ is replaced by a step $(1,-1)$ (resp. $(1,1)$) if $i$ is odd (resp. even). 
Moreover, for $i$ odd, the part $\gamma_i$ of the path between $(n_i +1,0)$ and $(n_{i+1},0)$ is replaced by $-\gamma_i -(0,1)$.

The inverse map is constructed as follows. Consider a  path of length $n$ from $(0,0)$ to $(n,N)$ with steps $(1,1)$ and $(1,-1)$. 
We change each step from $(n_i,0)$ to $(n_i+1,-1)$ to a step from $(n_i,0)$ to $(n_i+1,0)$. Each step from $(n_i,-1)$ to $(n_i+1,0)$ is changed to a step from $(n_i,0)$ to $(n_i+1,0)$. And each portion $\delta_i$ of the path between $(n_i+1,-1)$ and $(n_j,-1)$, $n_j>n_i$ located below the line $y=-1$ is changed to  $-\delta_i-(0,1)$. \end{proof}

\subsection{Parametrization of the generic strata} 

To parametrize the strata, we use a tool introduced in \cite{DES05} and later generalized for the non generic cases (for instance in \cite{BD10}), namely we transform the $\alpha\omega$-zones and rotation zones using the rectifying change of coordinate $z\mapsto t=\int \frac{dz} {iP_\eps(z)}$, endowed with the vector field $\frac{\partial}{\partial t}$.

\begin{proposition} Consider the image of an $\alpha\omega$-zone or rotation zone through $z\mapsto t=\int \frac{dz} {iP_\eps(z)}$.
\begin{enumerate} 
\item The image of a homoclinic loop is a horizontal segment of length $\kappa$: $\kappa$ measures the travel time along the homoclinic loop and is called the \emph{period} of the homoclinic loop. 
 \item The image of a rotation zone around a center $z_j$ is a vertical half-infinite strip (see Figure~\ref{strips_rotation-zones}). The half-infinite end is in the direction $i\R^+$ (resp. $i\R^-$) if the  period of the center given by $\int_{\gamma_j} \frac{dz} {iP_\eps(z)}=2\pi P'(z_j)$, is positive (resp. negative) (where $\gamma_j$ is a small loop around $z_j$ oriented in the positive direction and containing $z_j$ as the only singular point). The width of the strip is the absolute value of the period of $z_j$. The period could be either the period of a homoclinic loop or the sum of the periods of two homoclinic loops.
\item The image of an $\alpha\omega$-zone is a bi-infinite horizontal strip, whose width is given by the transversal time. This transversal time is uniquely defined in $\Hh$ if the $\alpha\omega$-zone does not intersect the real axis. It $\eta$ is the transversal time of such a zone, then $-\ov{\eta}$ is the transversal time of the mirror image of the zone with respect to the real axis. 
\item The transversal time of a bi-extremal zone is an element of $i\R^+$.
\item The transversal time of a right-extremal zone, left extremal zone or a central zone is not uniquely defined, since such a zone has three or four ends at infinity. Among these, only two have indices belonging to $\{0, 1, \dots, k\}$: we choose to define the transversal time as the travel time between these two ends. This transversal time is completely determined by the periods of the homoclinic loop(s) in the boundary of the zone and by the vertical width of the strip, i.e. a number in $i\R^+$. The vertical width represents the travel time along the part of the real axis contained in the zone (see Figure~\ref{transversal_extremal}). 
\end{enumerate}
\end{proposition}
\begin{proof} Only the last point requires a proof. 
 Let us consider a right-extremal zone. Then $\tau(0)= 2\ell-1$ for some $\ell>0$, and the upper transversal time is measured between the ends $e_0$ and $e_{2\ell-1}$ and called $\eta_{0,2\ell-1}$. The lower transversal time $\eta_{0,1-2\ell}$ is measured between the ends $e_0$ and $e_{1-2\ell}$. Let $\kappa\in \R^+$ be the travel time along the homoclinic loop bounding the extremal region. Then 
$\eta_{0,2\ell-1}=\eta_{0,1-2\ell}+\kappa.$
Since $\eta_{0,1-2\ell} = -\ov{\eta_{0,2\ell-1}}$ (see Figure~\ref{transversal_extremal}(a)), both $\eta_{0,2\ell-1}$ and $\eta_{0,1-2\ell}$ can be recovered from $\Im(e_{0,2\ell-1})$ and $\kappa$.
Two cases can occur for a left-extremal zone depending on the parity of $k$ (see Figure~\ref{transversal_extremal} (b) and (c)) and two cases can occur for a central zone  (see Figure~\ref{transversal_extremal} (d) and (e)). They are analyzed in the same way.\end{proof}
\begin{figure}\begin{center}\subfigure[]{\includegraphics[height=5cm]{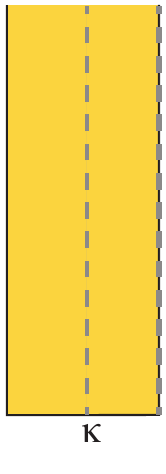}}\quad\subfigure[]{\includegraphics[height=5cm]{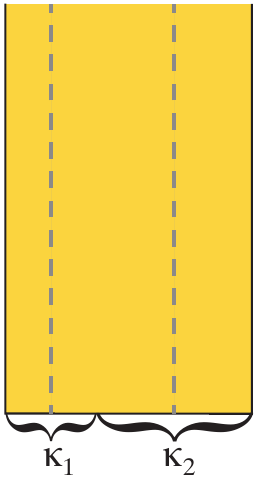}}
\quad\subfigure[]{\includegraphics[height=5cm]{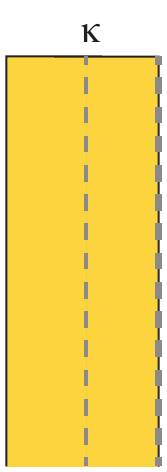}}\quad\subfigure[]{\includegraphics[height=5cm]{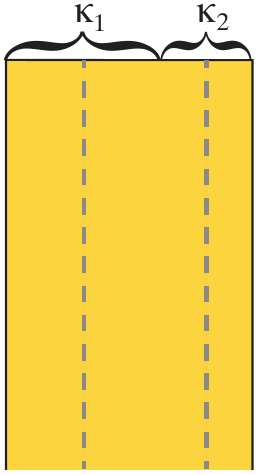}}\caption{The image of a rotation zone in the rectifying coordinate: in the positive direction for (a) and (b) and in the negative direction for (c) and (d). The boundary can consist of one (cases (a) and (c)) or two (cases (b) and (d)) homoclinic loops. The vertical dotted gray lines represent the images of the portions of the real line in the zone.}\label{strips_rotation-zones}\end{center}\end{figure}

\begin{figure}\begin{center}\subfigure[Right-extremal zone]{\includegraphics[width=3.5cm]{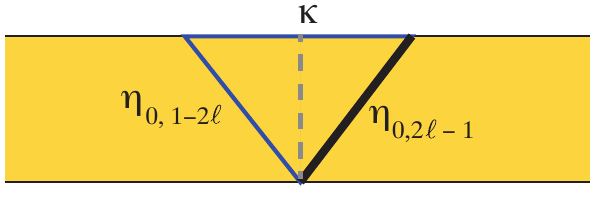}}\quad \subfigure[Left-extremal zone, $k$ odd]{\includegraphics[width=3.5cm]{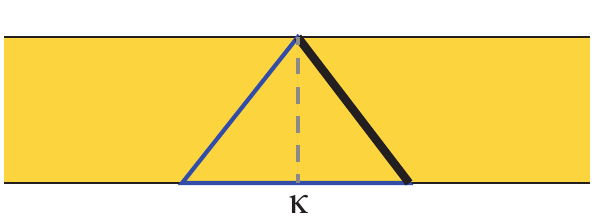}}\quad\subfigure[Left-extremal zone, $k$ even]{\includegraphics[width=3.5cm]{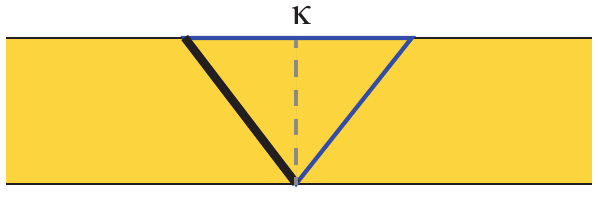}}\\
\subfigure[Central zone]{\includegraphics[width=3.5cm]{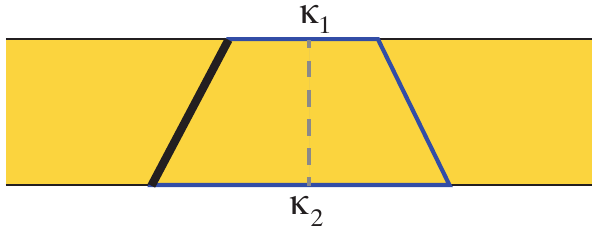}}\quad\subfigure[Central zone]{\includegraphics[width=3.5cm]{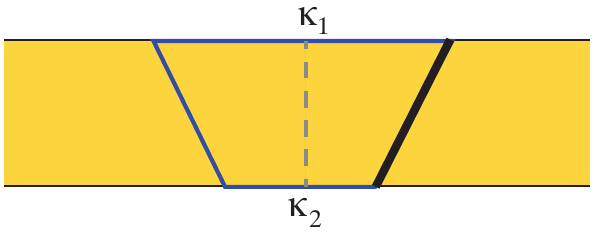}}\quad\subfigure[Bi-extremal zone]{\includegraphics[width=3.5cm]{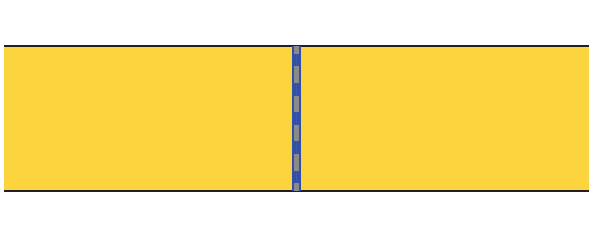}}\caption{The transversal times and periods of the homoclinic loops for right-extremal, left-extremal, central and bi-extremal zones. The transversal time just above the real axis is the bold segment. In the central cases, the lengths of $\kappa_1$ and $\kappa_2$ are arbitrary positive real numbers, and the transversal time is on the right (resp. left) if the homoclinic loop on the right part of the boundary is oriented upwards (resp. downwards). 
The vertical dotted gray lines represent the images of the portions of the real line.}\label{transversal_extremal}\end{center}\end{figure}

\begin{theorem}\label{thm:anal_inv} Consider a generic stratum inside $\mathcal{P}_{\rev, k+1}$. Let  $h$ be the number of homoclinic loops and let $m$ be the number of real singular points. Then the strata is parameterized by $(\R^+)^h\times (i\R^+)^{h+1-m} \times \Hh^{\frac{k-2h+m-1}2}$.
\end{theorem}
\begin{proof} There are $h+1-m$ regions containing an even number of complex singular points $2n_1$, \dots, $2n_{h+1-m}$. 
By the former discussion, the set of rotation zones and $\alpha\omega$-zones having at least one homoclinic loop in their boundary (and hence intersecting the real axis) are parametrized by $(\R^+)^h\times (i\R^+)^{h+1-m}$. It remains to parametrize the $\alpha\omega$-zones contained in the upper half-plane. Each $\alpha\omega$-zone is parametrized by one number in $\Hh$. There are $n_1+ \dots+ n_{h+1-m}- (h+1-m)$ such $\alpha\omega$-zones. Since $m+2(n_1+ \dots+ n_{h+1-m})=k+1$, then the number of such $\alpha\omega$-zones is $\frac{k-2h+m-1}2$. \end{proof}

\begin{remark} The total number of real parameters is $k$ as excepted for an open stratum in the $k$-dimensional parameter space $\eps$.\end{remark}

\begin{definition}\label{def:anal_inv} The element of $(\R^+)^h\times (i\R^+)^{h+1-m} \times \Hh^{\frac{k-2h+m-1}2}$ associated to a generic stratum in Theorem~\ref{thm:anal_inv} is called the \emph{analytic invariant} of the vector field. \end{definition}

\section{Structure of the family $\mathcal{P}_{\rev, k+1}$}

The bifurcations occuring in the family $\mathcal{P}_{\rev, k+1}$ are of the following  types:
\begin{enumerate}
\item Bifurcations of homoclinic loops through $\infty$, which are not symmetric with respect to the real axis. Such homoclinic loops occur in symmetric pairs (see Figure~\ref{bifurcation}(a)).  \item Bifurcation of parabolic singular points. The bifurcation of  multiple singular points of multiplicity $n+1$ occurs in a set of real codimension $n$ when the multiple singular point is on the real axis (see Figure~\ref{bifurcation}(b) and (c)), and of complex codimension $n$ (real codimension $2n$) otherwise (see Figure~\ref{bifurcation}(d)). We will show below that full unfoldings exist inside the family $\mathcal{P}_{\rev,k+1}$. \item Intersection of the former types.\end{enumerate} 
\begin{figure}\begin{center} 
\subfigure[Pair of homoclinic loops]{\includegraphics[width=5cm]{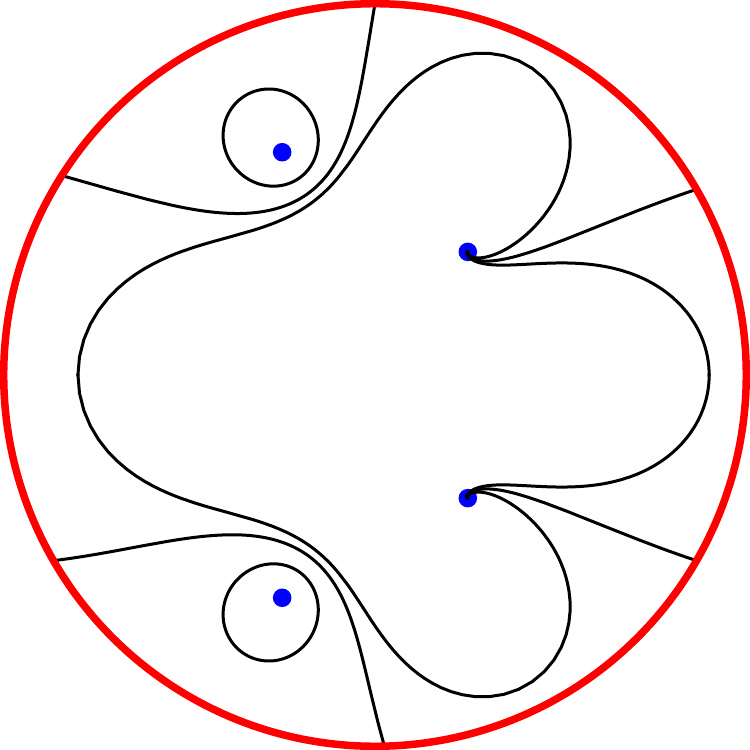}}\qquad\subfigure[Real double parabolic point]{\includegraphics[width=5cm]{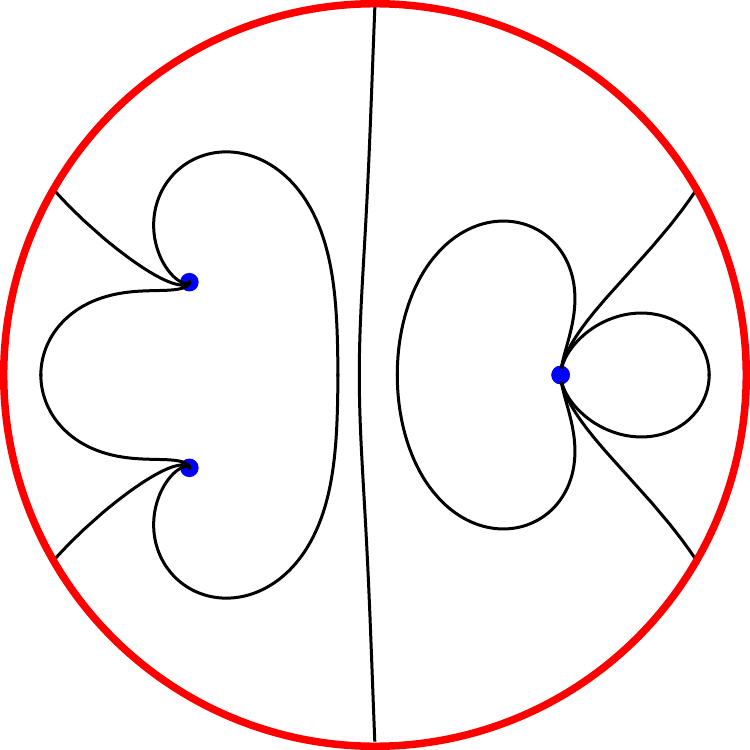}}
\subfigure[Real triple parabolic point]{\includegraphics[width=5cm]{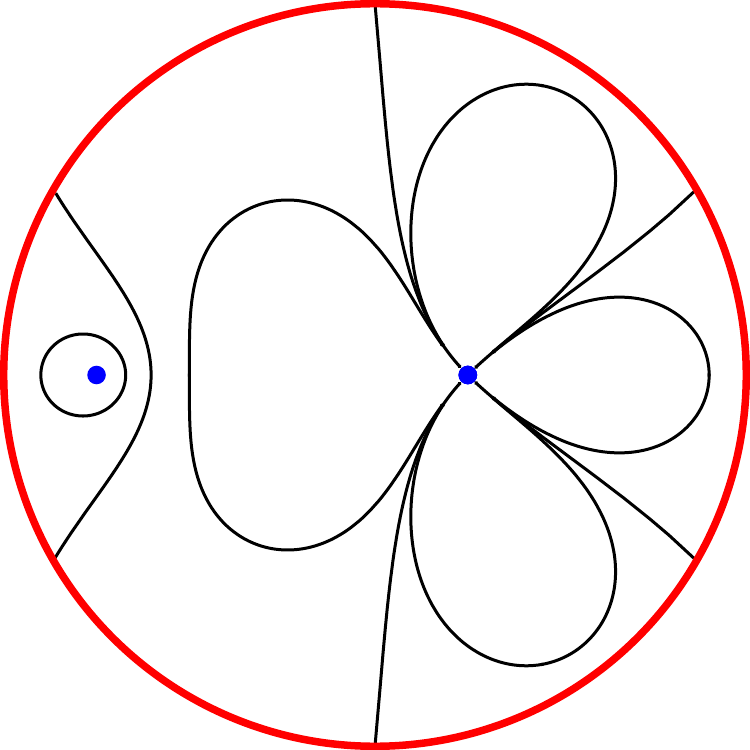}}\qquad \subfigure[Pair of complex double parabolic points]{\includegraphics[width=5cm]{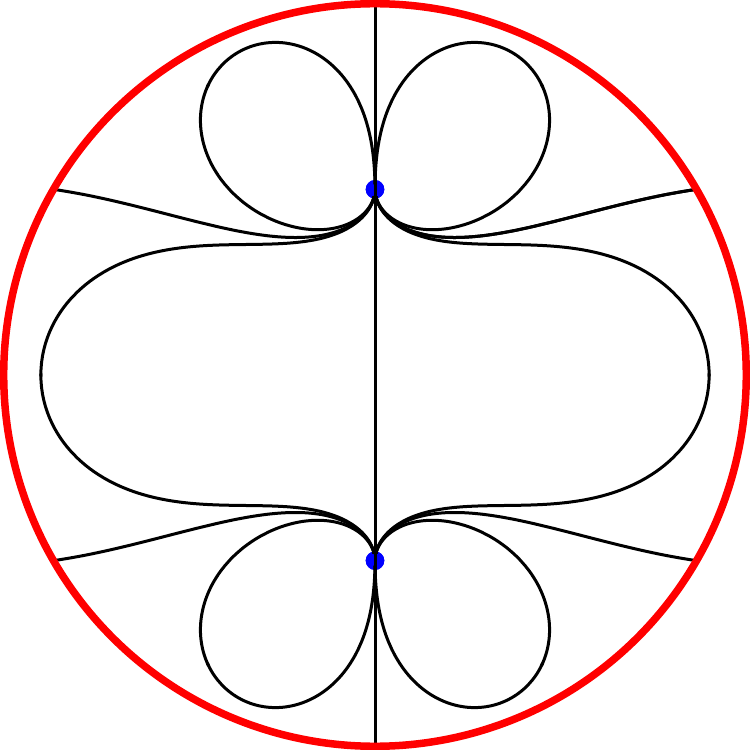}}
\caption{Bifurcations of real codimension 1 in (a) and (b), and real codimension 2 in (c) and (d).}\label{bifurcation}\end{center}\end{figure}

The family $\mathcal{P}_{\rev,k+1}$  is as complete as possible under the constraint that the system is reversible with respect to the real axis. 

\begin{theorem} Let $k \geq2$ and $z_0$ be a non real singular point of $iP_{\eps^*}(z) \frac{\partial}{\partial z}$ for a particular value $\eps^*$ of $\eps$. Then  a full unfolding exists in the family $\mathcal{P}_{\rev,k+1}$.\end{theorem}
\begin{proof} Let $s$ be the multiplicity of $z_0$: $z_0$ is simple is $s=1$ and multiple otherwise. Modulo a scaling, we can suppose that $z_0=a+i$, with $a\in \R$. 
Then $P_{\eps^*}(z) = (z-a-i)^s(z-a+i)^sQ_{\eps^*}(z)$, where ${\rm deg}(Q_{\eps^*})= \ell=k+1-2s$ and $Q_{\eps^*}$ is a polynomial with real coefficients. 

Let us first consider the case $\ell>0$. Let $\eps(\nu)=\eps^*+\nu$, for $\nu=(\delta_0, \eta_0,\dots, \delta_{s-1}, \eta_{s-1},\linebreak[1] \mu_0, \dots \mu_{\ell-2})\in \R^k$, and let us consider the following unfolding
$$P_{\eps(\nu)}(z) = R_{\eps(\nu)}(z) \ov{R}_{\eps(\nu)}(z) Q_{\eps(\nu)}(z),$$
with $$\begin{cases}
R_{\eps(\nu)}(z)=(z-a-i)^s + \sum_{j=0}^{s-1} (\delta_j+i\eta_j)(z-a-i)^j,\\
\ov{R}_{\eps(\nu)}(z)=(z-a+i)^s + \sum_{j=0}^{s-1} (\delta_j-i\eta_j)(z-a+i)^j,\\
Q_{\eps(\nu)}(z)= Q_{\eps^*}(z)-2\delta_{s-1}z^{\ell-1} +\sum_{j=0}^{\ell-2} \mu_jz^j.\end{cases}$$
Then $iP_{\eps(\nu)}\frac{\partial}{\partial z}\in \mathcal{P}$ and, from its form,  it is a complete unfolding of $iP_{\eps^*}\frac{\partial}{\partial z}$ in the neighborhood of $a+i$, since $\ov{R}_{\eps(\nu)}(z) Q_{\eps(\nu)}(z) = C+ O(\nu) +O(z-a-i)$, i.e. $R_{\eps(\nu)}(z)$ is the Weierstrass polynomial associated to $a+i$. 
By symmetry, $iP_{\eps(\nu)}\frac{\partial}{\partial z}$ is a complete unfolding of $iP_{\eps^*}\frac{\partial}{\partial z}$ in the neighborhood of $a-i$. Using induction on the degree of $P_{\eps^*}$, we can also conclude that $iP_{\eps(\nu)}\frac{\partial}{\partial z}$ is a complete unfolding of $iP_{\eps^*}\frac{\partial}{\partial z}$ in the neighborhood of each zero of $Q_{\eps^*}$.

Let us now consider the case $\ell=0$. In this case, necessarily $a=0$  and we consider the unfolding $P_{\eps(\nu)}(z) = R_{\eps(\nu)}(z) \ov{R}_{\eps(\nu)}(z)$ with
$$\begin{cases}
R_{\eps(\nu)}(z)=(z-i)^s +i\eta_{s-1}(z-i)^{s-1}+\sum_{j=0}^{s-2} (\delta_j+i\eta_j)(z-i)^j,\\
\ov{R}_{\eps(\nu)}(z)=(z+i)^s -i\eta_{s-1}(z-i)^{s-1}+ \sum_{j=0}^{s-2} (\delta_j-i\eta_j)(z+i)^j,\end{cases}$$ 
where $\nu=(\delta_0, \eta_0,\dots, \delta_{s-2}, \eta_{s-2}, \eta_{s-1})\in \R^k$.
Then $iP_{\eps(\nu)}\frac{\partial}{\partial z}$ is a complete unfolding of $iP_{\eps^*}\frac{\partial}{\partial z}$ in the neighborhood of $a\pm i$.
\end{proof}

\begin{remark} Consider $iP_\eps\frac{\partial}{\partial z}$. Since the eigenvalues of real singular points are pure imaginary,  and since the eigenvalues $\lambda_z$ and $\lambda _{\ov{z}}$ of complex conjugate singular points $z$ and $\ov{z}$ satisfy $\ov{\lambda_z}=-\lambda_{\ov{z}}$, then $k+1$ real parameters would be needed if the eigenvalues were to be independent. But  the eigenvalues of the singular points $z_1, \dots, z_{k+1}$ of  $iP_\eps\frac{\partial}{\partial z}$ satisfy $\sum_{j=1}^{k+1} \frac1{iP_\eps'(z_j)}=0$. Since this sum is pure imaginary, this is a real codimension 1 condition. This explains why a generic $iP_\eps\frac{\partial}{\partial z}$ depends on $k$ real parameters.\end{remark} 

\section{The case $k=2$} 
\begin{theorem} The bifurcation diagram of the vector field $i(z^3+\eps_1z+\eps_0)\frac{\partial}{\partial z}$ is given in Figure~\ref{Bif_diag-3}.
 \end{theorem}
\begin{figure} \begin{center} \includegraphics[width=10cm]{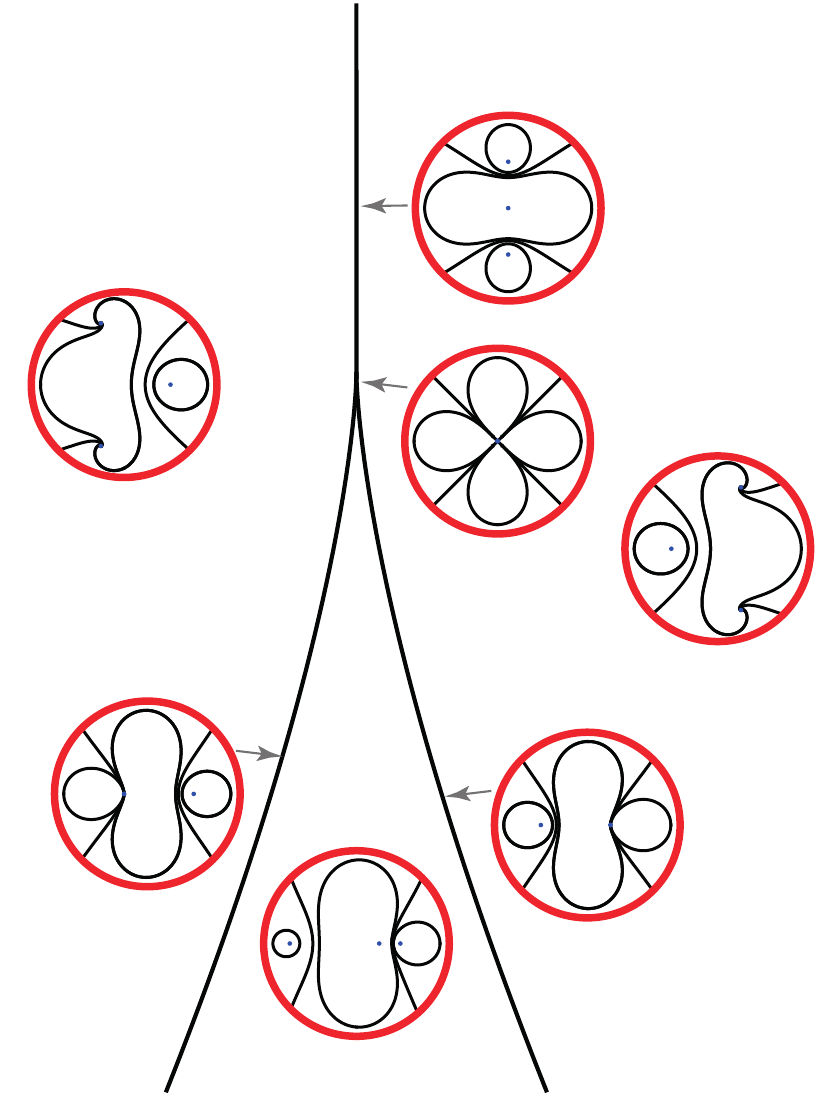}\caption{The bifurcation diagram of  $i(z^3+\eps_1z+\eps_0)\frac{\partial}{\partial z}$ with a sphere minus a point.}\label{Bif_diag-3}\end{center}\end{figure}
\begin{proof} Parabolic points occur along the discriminant curve $-4\eps_1^3-27\eps_0^2=0$. Moreover, a pair of complex singular points $a\pm ib$ can only be centers (and then surrounded by a homoclinic loop) if and only if $a=0$, i.e. $\eps_0=0$ and $\eps_1>0$. \end{proof}

\section{The case $k=3$} 
\begin{theorem} The bifurcation diagram of the vector field $iP(z)\frac{\partial}{\partial z}=i(z^4+\eps_2z^2+\eps_1z+\eps_0)\frac{\partial}{\partial z}$ has a conic structure using \eqref{scaling}. Its intersection with a sphere minus a point in the regular stratum with no real singular point is given in Figure~\ref{Bif_diag-4_2D}. 
\begin{figure} \begin{center} \includegraphics[width=15cm]{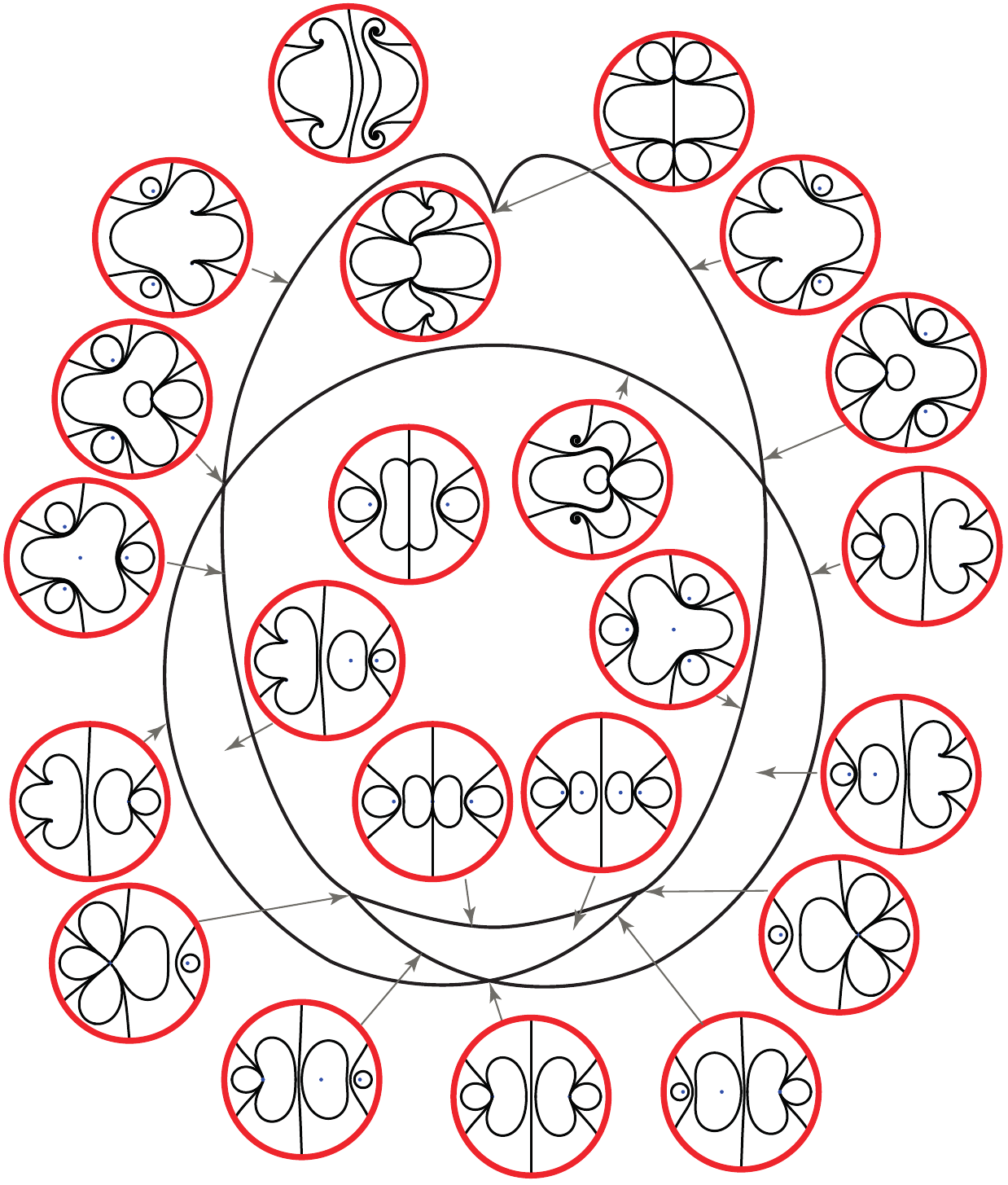}\caption{The intersection of the bifurcation diagram of  $i(z^4+\eps_2z^2+\eps_1z+\eps_0)\frac{\partial}{\partial z}$ with a sphere minus a point.}\label{Bif_diag-4_2D}\end{center}\end{figure}
The two codimension 1 bifurcations are of two types: 
\begin{enumerate}\item Existence of a real double parabolic point;
\item Existence of a pair of symmetric homoclinic loops, each surrounding a complex singular point. 
\end{enumerate}
 They occur on:
\begin{enumerate}\item The classical swallow tail which is part of discriminant locus
\begin{equation}\Delta= 256\eps_0^3-27\eps_1^4+144\eps_0\eps_1^2\eps_2-128\eps_0^2\eps_2^2-4\eps_1^2\eps_2^3+16\eps_0\eps_2^4=0.\label{eq_disc}\end{equation}
Note that the real algebraic variety $\Delta=0$ comprises a surface part of dimension 2 and the additional curve $\eps_1= 4\eps_0-\eps_2^2=0$, $\eps_2>0$. \item The part (H) of the surface \begin{equation}64 \eps_0^3 - 48 \eps_0^2 \eps_2^2 - 8\eps_1^2 \eps_2^3 + 12 \eps_0 \eps_2^4 - \eps_2^6=0,\label{sym_hom}\end{equation} 
$\eps_1\neq0$, which is located inside the region where there are at least two non real  singular points: a pair of symmetric centers surrounded by homoclinic loops occurs on (H). The surface \eqref{sym_hom} is regular outside the curve $\mathcal{C}=\{\eps_1= 4\eps_0-\eps_2^2=0\}$ (which is part of $\Delta=0$).  On the half of $\mathcal{C}$ for $\eps_2>0$,  there exists a pair of symmetric purely imaginary parabolic points and this half of $\mathcal{C}$  is part of the boundary of (H). There the surface \eqref{sym_hom} has a cuspidal edge. 
 \end{enumerate}
The two bifurcation surfaces also intersect at $4\eps_0-5\eps_2^2=\eps_1^2-8\eps_2^3$, where there is a real parabolic point, and at 
$27\eps_1^2+8\eps_2^3= 12\eps_0+\eps_2^2=0$, which is part of the boundary of $H$: there a triple real parabolic point occurs.\end{theorem}

\begin{proof} 
Note that any simple real singular point is a center and the boundary of its center bassin is a homoclinic loop or a pair of homoclinic loops, all symmetric with respect to the real axis. This occurs in open regions in parameter space and it is not a bifurcation.

The codimension 1 bifurcations are of two types: real parabolic points of multiplicity 2 and pairs of homoclinic loops symmetric with respect to the real axis. 
Higher order bifurcations occur at the intersection of codimension 1 bifurcations and also at higher multiplicity real parabolic points and at pairs of complex conjugate parabolic points. The latter two occur at the boundary of the codimension 1 bifurcation surfaces. Note that a complex double parabolic point has real codimension 2. 
\medskip
 
\noindent{\bf Bifurcations of parabolic points.} They occur when the discriminant $\Delta= 256\eps_0^3-27\eps_1^4+144\eps_0\eps_1^2\eps_2-128\eps_0^2\eps_2^2-4\eps_1^2\eps_2^3+16\eps_0\eps_2^4$ vanishes. The locus where $\Delta=0$ corresponding to multiple real roots is well known: it is the swallow tail. There is also the locus of nonreal pairs of parabolic points $(z_0,\ov{z}_0)$.
Because their sum vanishes, then $z_0=ia$ for some $a\in\R^*$. This occurs along the curve $\eps_1=0$, $\eps_2^2-4\eps_0=0$ and $\eps_0,\eps_2>0$. (Note that the other part of the curve $\eps_1=0$, $\eps_2^2-4\eps_0=0$ corresponds to a pair of real parabolic points: this is the self-intersection curve of the swallow tail.) 

\medskip
\noindent{\bf Bifurcations of homoclinic loops}. Homoclinic loops occur in symmetric pairs  when one complex singular point has a pure imaginary eigenvalue. Let $z_0=a+ ib$, $b\neq0$ be a singular point. Then $P'(a+ib)\in i\R$ if and only if $6a^2-2b^2+\eps_2=0$. This allows eliminating $b$ in the equation $P(a+ib)=0$, yielding the two equations
$$\begin{cases} -8 a^4 + \eps_0 + a \eps_1 - 2 a^2 \eps_2 - \eps_2^2/4=0,\\
-8a^3+\eps_1=0.\end{cases}$$
The resultant of the two equations with respect to $a$ vanishes exactly when \eqref{sym_hom} is valid. 

Let us now see that the surface \eqref{sym_hom}  has a cupsidal edge on the half of $\mathcal{C}$ for $\eps_2>0$. Because of the conic structure we can study the neighborhood of $(\eps_0,\eps_1,\eps_2)=(1,0,2)$, and we can cut along a plane $\eps_0=1$. Let $\eta_2=\eps_2-2$, and let $A=64 \eps_0^3 - 48 \eps_0^2 \eps_2^2 - 8\eps_1^2 \eps_2^3 + 12 \eps_0 \eps_2^4 - \eps_2^6$. Then $A|_{\eps_0=1}=\mathbf{-64(\eps_1^2-\eta_2^3)} + O(\eps_1^2\eta_2)+ O(\eta_2^4)$, yielding that the singularity is a cusp. 
\end{proof} 

\section{Realization}\label{sec:realization}

We have shown that to any generic vector field in $\mathcal{P}_{\rev,k+1}$, we can associate a combinatorial invariant (see Definition~\ref{def:comb_inv}) and an analytic invariant (see Definition~\ref{def:anal_inv}). In this section, we show the converse.

\begin{theorem}\label{thm_realization}
Let $\tau$ be
a non-crossing involution on $\{0, 1, \dots, k\}$, which preserves intervals between fixed points. Let $m$ be the number of fixed points of $\tau$, and let $h=q-1$, where $\{0, 1, \dots, k\}= A_1\cup \dots \cup A_q$ is a decomposition into minimal invariant subsets as in Proposition~\ref{prop:attachment}.
And let  $\eta\in (\R^+)^h\times (i\R^+)^{h+1-m} \times \Hh^{\frac{k-2h+m-1}2}$,

Then there exists a unique vector field $v_\eps\in \mathcal{P}_{\rev,k+1}$, which has $\tau$ as combinatorial invariant and $\eta$ as analytic invariant. 
\end{theorem} 
\begin{proof} The proof is standard. If the vector field  $v_\eps$ would exist, then the combinatorial and analytic invariants would give a description of the $\alpha\omega$-zones and rotation zones as strips and half-strips in the rectifying coordinate. The boundary of each whole strip is a union of half-lines and segments called $S_j$ (where $S_j$ represents the separatrix $s_j$). The combinatorial invariant would also describe the glueing of the strips: see Proposition~\ref{prop:attachment} and Figure~\ref{strips}. (Note that in the case of the vertical half-strips, the vertical boundaries are somewhat artificial and have no intrinsic dynamic character, except for being flow lines of the orthogonal vector field through the end point at infinity included in the rotation zone.) 

The vector field on the $\alpha\omega$-zones and rotation zones is conformally equivalent to  the vector field $\frac{\partial}{\partial t}$ on the strips and half-strips. 
What is important is to add to this description the reversible character of the vector field. For the $\alpha\omega$-zones and rotation zones intersecting the real axis, the real axis is sent to vertical segments or half-lines, and the vector field $\frac{\partial}{\partial t}$ is indeed reversible with respect to these segments or half-lines. 
The zones not intersecting the real axis come in pairs symmetric one of the other: they have the same vertical width, and the corresponding transversal times are transformed one into the other through $t\mapsto -\ov{t}$.

Hence, starting with the strips equiped with the vector field $\frac{\partial}{\partial t}$, we construct an abstract manifold by glueing the strips. The manifold has $k+1$ infinite ends conformally equivalent to half-cylinders, and the number of $S_j$ in each half-cylinder is given by Proposition~\ref{prop:attachment}. This manifold can be shown to be confirmally equivalent to $\CP^1$ punctured in $k+1$ points together with a vector field wih one pole of order $k-1$. The vector field can be extended to the punctured points where it will have simple singular points. Choosing a coordinate so that the pole be at $\infty\in \CP^1$, the vector field is polynomial. It is uniquely defined up to a rotation if a real scaling brings the leading coefficient to have modulus $1$, and a translation brings the sum of roots to zero. A complete proof of this for monic vector fields can be found in \cite{DES05} for DES-generic vector fields and in \cite{BD10} for the general case. 

The only thing to prove is that the vector field is reversible with respect to the real axis. We have an antiholomorphic involution $t\mapsto -\ov{t}$ on the strips and half-strips containing an image of a portion of the real axis. From the construction,  it is possible to extend the involution in an antiholomorphic way in the other strips and half-strips and the vector field remains reversible with respect to it. Since the involution is bounded at the singular points, it can be extended antiholomorphically to these points. 
From its global character, this involution is the Schwarz reflection with respect to a line or circle. Because the curve passes through the pole, it is a line. Moreover, we decide to orient this line towards the end point $e_0$. (Note that $E_0$ is always on the bottom of the boundary of the strip or half-strip to which its belongs.) Using a rotation, we can suppose that the line is the real axis, and that it is oriented to the right. The fact that the leading coefficient is given by $i$ comes from the fact that, near $e_0$, the angle between $\R^+$ and the vector field is $+\pi/2$.
\end{proof}

\begin{figure}\begin{center} 
\subfigure[A rotation zone and five $\alpha\omega$-zones]{\includegraphics[width=5cm]{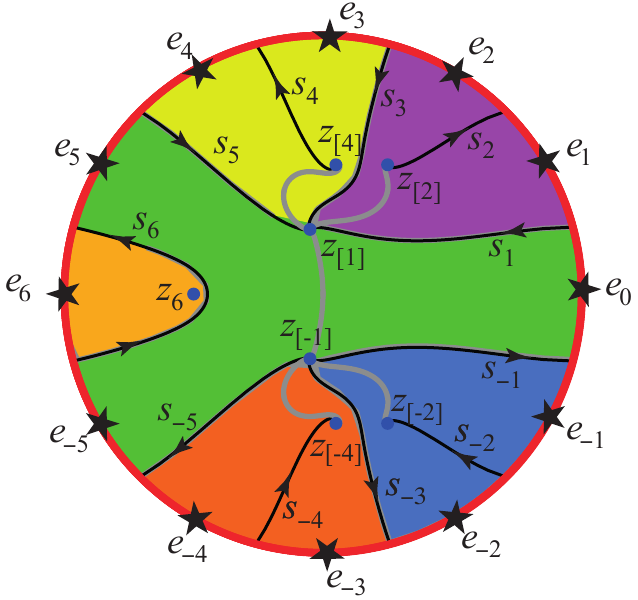}}\qquad\quad\subfigure[Their images in the rectifying coordinate]{\includegraphics[width=6.5cm]{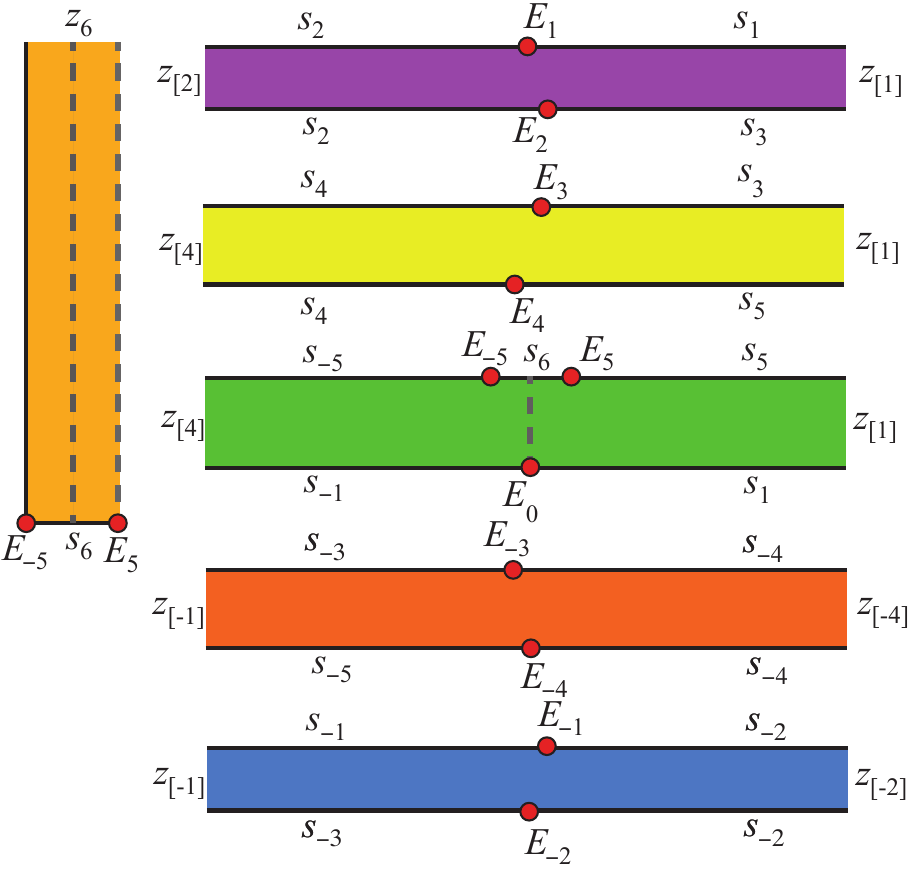}}\caption{A rotation zone, five $\alpha\omega$-zones and their images as half-strip and strips in the rectifying coordinate. The glueing is done along separatrices noted $s_{\pm j}$ in both coordinates. The ends $e_j$ on the left correspond to the points $E_j$ on the right. The gray dotted lines on the right are the images of pieces of the real axis.}\label{strips}\end{center}\end{figure}

\end{document}